\documentclass[12pt]{amsart}
\usepackage{amsmath}
\usepackage{amsfonts}
\usepackage{amssymb}
\usepackage[all,cmtip]{xy}           

\usepackage{bbm}
\usepackage{bbding}
\usepackage{txfonts}
\usepackage[shortlabels]{enumitem}
\usepackage{ifpdf}
\ifpdf
\usepackage[colorlinks,final,backref=page,hyperindex]{hyperref}
\else
\usepackage[colorlinks,final,backref=page,hyperindex,hypertex]{hyperref}
\fi
\usepackage{tikz-cd}
\usepackage[active]{srcltx}
\usepackage{tikz}
\usepackage{amscd}
\usepackage{bm}

\makeatletter

\newtheorem{thm}{Theorem}[section]
\newtheorem{prop}[thm]{Proposition}
\newtheorem{lem}[thm]{Lemma}
\newtheorem{cor}[thm]{Corollary}
\newtheorem{prop-def}{thm}[section]

\theoremstyle{definition}
\newtheorem{defn}[thm]{Definition}

\newtheorem{remark}[thm]{Remark}
\newtheorem{exam}[thm]{Example}

\newcommand{\nc}{\newcommand}

 \nc{\mbibitem}[1]{\bibitem{#1}} 
 \nc{\mrm}[1]{{\rm #1}}

\nc{\pl}{\cdot}
 \nc{\la}{\longrightarrow}
\nc{\ot}{\otimes}
 \nc{\rar}{\rightarrow}

\nc{\PLA}{{\mathrm{Leib}}}
\nc{\bfk}{{\bf k}}
\nc{\C}{{\mathrm{C}}}
\nc{\RBA}{\mathsf{MLeib}}
\nc{\RBO}{\mathsf{MO}}
\nc{\End}{\mrm{End}}
\nc{\Ext}{\mrm{Ext}}
\nc{\Fil}{\mrm{Fil}}
\nc{\Fr}{\mrm{Fr}}
\nc{\Frob}{\mrm{Frob}}
\nc{\Gal}{\mrm{Gal}}
\nc{\GL}{\mrm{GL}}
\nc{\Hom}{\mrm{Hom}}
\nc{\Hoch}{\mrm{Hoch}}
\nc{\hsr}{\mrm{H}}
\nc{\hpol}{\mrm{HP}}
\nc{\im}{\mrm{im}}
\nc{\Id}{\mrm{Id}}
\nc{\Irr}{\mrm{Irr}}
\nc{\incl}{\mrm{incl}}
\nc{\length}{\mrm{length}}
\nc{\NLSW}{\mrm{NLSW}}
\nc{\Lie}{\mrm{Lie}}
\nc{\mchar}{\rm char}
\nc{\mpart}{\mrm{part}}
\nc{\ql}{{\QQ_\ell}}
\nc{\qp}{{\QQ_p}}
\nc{\rank}{\mrm{rank}}
\nc{\rcot}{\mrm{cot}}
\nc{\rdef}{\mrm{def}}
\nc{\rdiv}{{\rm div}}
\nc{\rmH}{ {\mathrm{H}}}
\nc{\rtf}{{\rm tf}}
\nc{\rtor}{{\rm tor}}
\nc{\res}{\mrm{res}}
\nc{\Sh}{{\mathrm{Sh}}}
\nc{\SL}{\mrm{SL}}
\nc{\Spec}{\mrm{Spec}}
\nc{\sgn}{{\mathrm{sgn}}}
\nc{\tor}{\mrm{tor}}
\nc{\Tr}{\mrm{Tr}}
\nc{\tr}{\mrm{tr}}
\nc{\wt}{\mrm{wt}}
\nc{\op}{\mrm{op}}

\nc{\BA}{{\mathbb A}}   \nc{\CC}{{\mathbb C}}
\nc{\DD}{{\mathbb D}}   \nc{\EE}{{\mathbb E}}
\nc{\FF}{{\mathbb F}}   \nc{\GG}{{\mathbb G}}
\nc{\HH}{ \mathrm{HH}}   \nc{\LL}{{\mathbb L}}
\nc{\NN}{{\mathbb N}}   \nc{\PP}{{\mathbb P}}
\nc{\QQ}{{\mathbb Q}}   \nc{\RR}{{\mathbb R}}
\nc{\TT}{{\mathbb T}}   \nc{\VV}{{\mathbb V}}
\nc{\ZZ}{{\mathbb Z}}   \nc{\TP}{\widetilde{P}}
\nc{\m}{{\mathbbm m}}

\nc{\cala}{{\mathcal A}}    \nc{\calc}{{\mathcal C}}
\nc{\cald}{\mathcal{D}}     \nc{\cale}{{\mathcal E}}
\nc{\calf}{{\mathcal F}}    \nc{\calg}{{\mathcal G}}
\nc{\calh}{{\mathcal H}}    \nc{\cali}{{\mathcal I}}
\nc{\call}{{\mathcal L}}    \nc{\calm}{{\mathcal M}}
\nc{\caln}{{\mathcal N}}    \nc{\calo}{{\mathcal O}}
\nc{\calp}{{\mathcal P}}    \nc{\calr}{{\mathcal R}}
\nc{\cals}{{\mathcal S}}    \nc{\calt}{{\Omega}}
\nc{\calv}{{\mathcal V}}    \nc{\calw}{{\mathcal W}}
\nc{\calx}{{\mathcal X}}

\nc{\fraka}{{\mathfrak a}}
\nc{\frakb}{\mathfrak{b}}
\nc{\frakg}{{\frak g}}
\nc{\frakl}{{\frak l}}
\nc{\fraks}{{\frak s}}
\nc{\frakh}{{\frak h}}
\nc{\frakm}{{\frak m}}
\nc{\frakM}{{\frak M}}
\nc{\frakp}{{\frak p}}
\nc{\frakW}{{\frak W}}
\nc{\frakX}{{\frak X}}
\nc{\frakS}{{\frak S}}
\nc{\frakA}{{\frak A}}
\nc{\frakC}{{\frak{C}}}
\nc{\frakx}{{\frakx}}

\nc{\lir}[1]{\textcolor{red}{\underline{Li:}#1 }}

\begin{document}

\title[ Crossed homomorphisms on Leibniz algebras]{ Cohomology and Deformation theory of crossed homomorphisms on Leibniz algebras}

\author{Yizheng Li}
\address{
 School of Mathematical Sciences, Qufu Normal University,  Qufu 273165, P. R. of China}
\email{yzli1001@163.com}

\author{$\mbox{Dingguo\ Wang}^1$}\footnote{Corresponding author}
\address{School of Mathematical Sciences, Qufu Normal University,  Qufu 273165, P. R. of China}
\email{dgwang@qfnu.edu.cn}

\date{\today}

\begin{abstract}
 In this paper,  we construct a differential graded Lie algebra whose Maurer-Cartan elements
are given by crossed homomorphisms on Leibniz algebras. This allows us to define cohomology
for a crossed homomorphism. Finally, we study linear deformations, formal deformations and extendibility of finite order deformations of a crossed homomorphism in terms of the cohomology theory.
\end{abstract}

\subjclass[2010]{
16E40   
16S80   
17A30 
}

\keywords{crossed homomorphism, cohomology,  Maurer-Cartan elements, deformation}

\maketitle

\tableofcontents

\allowdisplaybreaks

\section*{Introduction}

In 1960, Baxter \cite{B60} introduced the notion of Rota-Baxter operators on associative algebras in his study of fluctuation theory in probability. Rota-Baxter operators have been found many applications, including in Connes-Kreimer's algebraic approach
to the renormalization in perturbative quantum field theory \cite{CK00}.
For more details on the Rota-Baxter operator, see \cite{G12}.

The notion of crossed homomorphisms of Lie algebras was first introduced by Lue \cite{Lue} in the study of non-abelian extensions of Lie algebras.  A crossed homomorphism is nothing but a differential
operator of weight 1. The authors  showed that the category of weak representations
(resp. admissible representations) of Lie-Rinehart algebras (resp. Leibniz pairs) is a left module
category over the monoidal category of representations of Lie algebras using crossed homomorphisms \cite{PST}. Recently, the author considerd crossed homomorphisms between associative algebras \cite{Das}.

The concept of Leibniz algebras was introduced  by Loday \cite{L93, LP93} in the study of the algebraic $K$-theory. Relative Rota-Baxter
operators on Leibniz algebras were studied in \cite{ST2}, which is the main ingredient in the study of the twisting theory and the bialgebra theory for Leibniz algebras \cite{ST1}. Recently,  the author has considered  weighted relative Rota-Baxter operators  on Leibniz algebras  in \cite{Das1}.   Our aim in this paper is to consider crossed homomorphisms between Leibniz algebras using Leibniz representation  introduced by  \cite{Das1}, More  precisely, we  construct a differential graded Lie algebra whose Maurer-Cartan elements
are given by crossed homomorphisms on Leibniz algebras. This allows us to define cohomology
for a crossed homomorphism. Finally, we study linear deformations, formal deformations and extendibility of finite order deformations of a crossed homomorphism in terms of the cohomology theory.

The paper is organized as follows. In Section 1, we recall some basic definitions about  Leibniz algebras and their cohomology. In Section 2, we consider crossed homomorphisms between Leibniz algebras.  In Section 3, we construct a differential graded Lie algebra whose Maurer-Cartan elements
are given by crossed homomorphisms on Leibniz algebras and define cohomology
for a crossed homomorphism.  In Section 4, we study linear deformations, formal deformations and extendibility of finite order deformations of a crossed homomorphism in terms of the cohomology theory.

 Throughout this paper, let $\bfk$ be a field of characteristic $0$.  Except specially stated,  vector spaces are  $\bfk$-vector spaces and  all    tensor products are taken over $\bfk$.

\section{Leibniz algebras, representations and their cohomology theory}\
We start with the background of Leibniz algebras and their cohomology that we refer the reader to~ \cite{cuvier,LP93, Das1} for more details.

\begin{defn}
A Leibniz algebra is  a vector space $\frakg$ together with a bilinear operation (called bracket) $[\cdot, \cdot]_\frakg : \frakg \otimes \frakg \rightarrow  \frakg$  satisfying
\begin{eqnarray*}
[x, [y, z]_\frakg]_\frakg=[[x, y]_\frakg, z]_\frakg+[y, [x, z]_\frakg]_\frakg,~ \text{ for } x, y, z \in \frakg.
\end{eqnarray*}
\end{defn}

A Leibniz algebra as above may be denoted by the pair $(\frakg, [\cdot, \cdot])$ or simply by $\frakg$ when no confusion arises. A Leibniz algebra whose bilinear bracket is skewsymmetric is nothing but a Lie algebra. Thus, Leibniz algebras are the non-skewsymmetric analogue of Lie algebras.

\begin{defn}
A representation of a Leibniz algebra $(\frakg, [\cdot, \cdot]_\frakg)$ consists of a triple $(V, \rho^L, \rho^R)$ in which $V$ is a vector space and $\rho^L, \rho^R : \frakg \rightarrow gl(V)$ are linear maps satisfying for $x, y \in \frakg$,
\begin{align*}
\begin{cases}
\rho^L ([x,y]_\frakg) = \rho^L (x)  \rho^L (y) - \rho^L (y)  \rho^L (x),\\
\rho^R ([x,y]_\frakg) = \rho^L (x) \rho^R (y) - \rho^R (y)  \rho^L (x),\\
\rho^R ([x,y]_\frakg) = \rho^L (x)  \rho^R (y) + \rho^R (y)  \rho^R (x).
\end{cases}
\end{align*}
\end{defn}

It follows that any Leibniz algebra $\frakg$ is a representation of itself with
\begin{align*}
    \rho^L (x) = L_x = [x, ~]_\frakg ~~ \text{ and } ~~ \rho^R (x) = R_x = [~, x]_\frakg, \text{ for } x \in \frakg.
\end{align*}
Here $L_x$ and $R_x$ denotes the left and right multiplications by $x$, respectively. This is called the regular representation.

\medskip

Let $(\frakg, [\cdot, \cdot]_\frakg)$ be a Leibniz algebra and  $(V, \rho^L, \rho^R)$ be a representation of it. The cohomology of the Leibniz algebra $\frakg$ with coefficients in $V$ is the
cohomology of the cochain complex $\{ C^\ast(\frakg, V), \delta \}$, where $C^n (\mathfrak{g}, V) = \mathrm{Hom}(\frakg^{\otimes n}, V )$ for $n \geq 0$, and the
coboundary operator $\delta:C^n(\frakg, V ) \rightarrow C^{n+1}(\frakg, V )$ given by
\begin{align*}
&(\delta f)(x_1,\ldots, x_{n+1})\\
&=\sum^{n}_{i=1}(-1)^{i+1}\rho^L(x_i)f(x_1,\ldots, \hat{x}_i, \ldots, x_{n+1}) + (-1)^{n+1}\rho^R(x_{n+1})f(x_1,\ldots,  x_{n})\\
&+\sum_{1\leq i< j\leq n+1} (-1)^{i}f(x_1,\ldots, \hat{x}_i, \ldots, x_{j-1}, [x_i, x_j]_\frakg, x_{j+1},\ldots, x_{n+1}),
\end{align*}
for $x_1, \ldots, x_{n+1}\in \frakg$. The corresponding cohomology groups are denoted by $H^{\ast}(\frakg, V).$ This cohomology has been first appeared in \cite{cuvier} and rediscover by Loday and Pirashvili \cite{LP93}. This cohomology is also the Loday-Pirashvili cohomology.

\begin{defn}(\cite{B97, FM08})  The graded vector space $C^{\ast}(\mathfrak{g}, \mathfrak{g})$
 equipped with  Balavoine bracket
\begin{eqnarray*}
\llbracket P, Q\rrbracket := P \overline{\diamond} Q - (-1)^{pq}Q \overline{\diamond} P ~~~~~~~~~\forall P \in C^{p+1}(\mathfrak{g}, \mathfrak{g}), Q \in C^{q+1}(\mathfrak{g}, \mathfrak{g})
\end{eqnarray*}
is a graded Lie algebra, where $P \overline{\diamond} Q \in C^{p+q+1}(\mathfrak{g}, \mathfrak{g})$ is defined by
\begin{eqnarray*}
P \overline{\diamond} Q =\sum_{k=1}^{p+1}(-1)^{(k-1)q}P\diamond_{\mathrm{\bfk}} Q
\end{eqnarray*}
and $\diamond_{\mathrm{\bfk}}$ is defined by
\begin{eqnarray*}
&&P \diamond_{\mathrm{\bfk}} Q(x_1, \cdots , x_{p+q+1})\\
&=& \sum_{\sigma\in \mathbb{S}_{(k-1, q)}}(-1)^{\sigma}P(x_{\sigma(1)},\cdots , x_{\sigma(k-1)}, Q(x_{\sigma(k)},\cdots , x_{\sigma(k+q-1)}, x_{k+q}), x_{k+q+1}, \cdots,x_{p+q+1} ),
\end{eqnarray*}
for all  $x_1, \cdots, x_{p+q+1}\in \mathfrak{g}$.

Moreover, $\mu_\mathfrak{g}:  \mathfrak{g}\otimes \mathfrak{g}\rightarrow \mathfrak{g}$ is  a Leibniz bracket if and only if $\llbracket\mu_\mathfrak{g}, \mu_\mathfrak{g}\rrbracket= 0$, i. e. $\mu_\mathfrak{g}$is a Maurer
\\ -Cartan element of the graded Lie algebra $(C^{\ast}(\mathfrak{g}, \mathfrak{g}), \llbracket-, -\rrbracket)$.
\end{defn}

\begin{defn}
Let $(\frakg, [\cdot, \cdot]_\frakg)$ and $(\frakh, [\cdot, \cdot]_\frakh)$ be two Leibniz algebras. We say that $\frakh$ is a Leibniz $\frakg$-representation if there are bilinear
maps $\rho^L, \rho^R: \frakg \rightarrow$ End$(\frakh)$ that make $(\frakh, \rho^L, \rho^R)$ into a representation of the Leibniz algebra $\frakg$
satisfying additionally
\begin{eqnarray*}
&&  (La)\qquad\rho^{L}(x)[h, k]_\frakh=[\rho^{L}(x)h, k]_\frakh+[h, \rho^{L}(x)k]_\frakh,\\
&& (Lb)\qquad [h, \rho^{R}(x)k]_\frakh=\rho^{R}(x)[h, k]_\frakh+[k, \rho^{R}(x)h]_\frakh,\\
&& (Lc)\qquad [h, \rho^{L}(x)k]_\frakh=[\rho^{R}(x)h, k]_\frakh+\rho^{L}(x)[h, k]_\frakh,
\end{eqnarray*}
for $h, k\in \frakh, x \in \frakg$.
\end{defn}
Note that, for any Leibniz algebra $(\frakg, [\cdot, \cdot]_\frakg)$, the regular representation is a Leibniz $\frakg$-representation.


\section{Crossed homomorphisms on  Leibniz algebras}

In  this section, we study  crossed homomorphisms between  Leibniz algebras.

\begin{defn}\cite{Das1}\label{Def: crossed homomorphisms} Let $(\frakg, [\cdot, \cdot]_\frakg)$ and $(\frakh, [\cdot, \cdot]_\frakh)$ be Leibniz algebras  and $(\frakh, \rho^L, \rho^R)$ be a Leibniz $\frakg$-representation. If a linear map $H:\frakg\rightarrow \frakh$ is said to be a  crossed homomorphism from $\frakg$ to $\frakh$ such that the following equation
\begin{eqnarray}\label{Def: crossed homomorphisms}
H([x, y]_\frakg)=\rho^L(x)H(y)+\rho^R(y)H(x)+[H(x), H(y)]_\frakh
\end{eqnarray}
holds for any  $x, y \in \frakg$.
\end{defn}

\begin{remark}
A crossed homomorphism from $\frakg$ to $\frakg$ with respect to the regular representation is also called
a differential operator of weight 1.
\end{remark}

\begin{exam}
If the action $\rho^L, \rho^R$ of $\frakg$ on $\frakh$ is zero, then any crossed homomorphism from $\frakg$ to $\frakh$ is
nothing but a Leibniz algebra homomorphism.
\end{exam}

\begin{defn}
Let $H, H': \frakg \rightarrow \frakh$ be two crossed homomorphisms from $\frakg$ to $\frakh$.
A morphism from $H$ to $H'$ consists of two Leibniz algebra morphisms $\phi_\frakg: \frakg \rightarrow \frakg$ and
$\phi_\frakh: \frakh \rightarrow \frakh$ satisfying $\phi_\frakh \circ H = H' \circ \phi_\frakg$, $\phi_\frakh(\rho^L(x)h) = \rho^L(\phi_\frakg(x))\phi_\frakh(h)$ and $\phi_\frakh(\rho^R(x)h) = \rho^R(\phi_\frakg(x))\phi_\frakh(h)$, for all $x\in \frakg$ and $h\in \frakh$.
\end{defn}

One can  construct a new Leibniz $\frakg$-representation.

\begin{lem}\label{Lem: new-representation}
Let $H: \frakg \rightarrow \frakh$ be a crossed homomorphism. Define maps $\rho_H^L, \rho_H^R: \frakg \rightarrow$ End$(\frakh)$ by
\begin{eqnarray} \label{new-equations}
&&\rho^L_H(x)h:=\rho^L(x)h+[H(x), h]_\frakh,\nonumber\\
&& \rho^R_H(x)h:=\rho^R(x)h+[h, H(x)]_\frakh,
\end{eqnarray}
for $x\in \frakg$ and $h\in \frakh$. Then $(\frakh, \rho^L_H, \rho^R_H)$ is a Leibniz $\frakg$-representation.
\end{lem}

\begin{proof}

First we  prove that $\rho_H^L, \rho_H^R $ satisfy the conditions L(a)-L(c) as follows.

For any $x\in \frakg$ and $h, k\in \frakh$, we have
\begin{eqnarray*}
(La')\qquad\qquad&&[\rho^{L}_H(x)h, k]_\frakh+[h, \rho^{L}_H(x)k]_\frakh\\
&=&  [\rho^{L}(x)h, k]_\frakh +[[H(x), h]_\frakh, k]_\frakh +[h, \rho^{L}(x)k]_\frakh+[h, [H(x), k]_\frakh]_\frakh\\
&=& \rho^L(x)[h, k]_\frakh+[H(x), [h, k]_\frakh]_\frakh\\
&=& \rho^L_H(x)[h, k]_\frakh.
\end{eqnarray*}
Similar to prove that
\begin{eqnarray*}
(Lb')\qquad\qquad [h, \rho_H^{R}(x)k]_\frakh=\rho_H^{R}(x)[h, k]_\frakh+[k, \rho_H^{R}(x)h]_\frakh.
\end{eqnarray*}
Furthermore, we have
\begin{eqnarray*}
(Lc')\qquad\qquad &&[\rho_H^{R}(x)h, k]_\frakh+\rho_H^{L}(x)[h, k]_\frakh\\
 &=&[\rho^{R}(x)h, k]_\frakh+[[h, H(x)]_\frakh, k]_\frakh+\rho^{L}(x)[h, k]_\frakh + [H(x),[h, k]_\frakh]_\frakh  \\
&=& [h, \rho^{L}(x)k]_\frakh+[h, [H(x), k]_\frakh]_\frakh\\
 &=& [h, \rho^{L}_H(x)k]_\frakh.
\end{eqnarray*}
Next we prove that $ (\frakh, \rho_H^L, \rho_H^R) $  is a representation over $(\frakg, [\cdot, \cdot]_\frakg)$.

For any $x, y\in \frakg$ and $h\in \frakh$, we have
\begin{eqnarray*}
&& \rho_H^L (x)  \rho_H^L (y)h - \rho_H^L (y)  \rho_H^L (x)h\\
&=& \rho_H^L (x) (\rho^L (y)h+[H(y), h]_\frakh)-\rho_H^L (y) (\rho^L (x)h+[H(x), h]_\frakh)\\
&=& \rho^L (x)\rho^L (y)h+[H(x), \rho^L (y)h]_\frakh+\rho^L (x)[H(y), h]_\frakh +[H(x), [H(y), h]_\frakh]_\frakh \\
&&-\rho^L (y)\rho^L (x)h-[H(y), \rho^L (x)h]_\frakh -\rho^L (y)[H(x), h]_\frakh -[H(y), [H(x), h]_\frakh]_\frakh \\
&\stackrel{(La, Lc)}{=}& \rho^L ([x, y]_\frakg)h+ [\rho^L (y)H(x), h]_\frakh+[\rho^R (x)H(y), h]_\frakh +[H(x), H(y)]_\frakh, h]_\frakh \\
&=& \rho^L ([x, y]_\frakg)h +[H([x, y]_\frakg), h]_\frakh\\
&=& \rho^L_H ([x, y]_\frakg)h.
\end{eqnarray*}
Similarly, we have
\begin{eqnarray*}
\rho_H^R([x, y]_\frakg)=\rho_H^L (x)  \rho_H^R (y) - \rho_H^R (y)  \rho_H^L (x).
\end{eqnarray*}
 Furthermore, we have
\begin{eqnarray*}
 &&\rho_H^R(y) \rho_H^L(x)h+\rho_H^R(y) \rho_H^R(x)h\\
 &=& \rho_H^R(y)(\rho^L(x)h+[H(x), h]_\frakh)+\rho_H^R(y) (\rho^R(x)h+[h, H(x)]_\frakh)\\
 &=& \rho^R(y)\rho^L(x)h+[\rho^L(x)h, H(y)]_\frakh+ \rho^R_H(y)[H(x), h]_\frakh \\
 && +\rho^R(y)\rho^R(x)h +[\rho^R(x)h, H(y)]_\frakh +\rho^R_H(y)[h, H(x)]_\frakh\\
 &\stackrel{(La, Lb', Lc)}{=}& 0.
 \end{eqnarray*}
Hence,   $(\frakh, \rho^L_H, \rho^R_H)$ is a Leibniz $\frakg$-representation.
\end{proof}

Since $(\frakh, [\cdot, \cdot]_\frakh)$ is  a Leibniz $\frakg$-representation, there is a semi-direct product Leibniz algebra
structure on $\frakg \oplus\frakh$ given by
  \begin{eqnarray*}
[(x, h), (y, k)]=[x, y]_\frakg+\rho^L(x)k+\rho^R(y)h+[h, k]_\frakh,
  \end{eqnarray*}
for $(x, h),(y, k) \in \frakg \oplus\frakh$. We denote this semi-direct product algebra by $\frakg \ltimes \frakh$.
Moreover, it follows from Lemma \ref{Lem: new-representation} that the direct sum $\frakg \ltimes \frakh$ carries another
semi-direct product Leibniz algebra structure given by
  \begin{eqnarray*}
[(x, h), (y, k)]_H=[x, y]_\frakg+\rho^L_H(x)k+\rho^R_H(y)h+[h, k]_\frakh.
  \end{eqnarray*}
We denote this semi-direct product algebra by $\frakg \ltimes_H \frakh$.

\begin{thm}
Let $(\frakg, [\cdot, \cdot]_\frakg)$ and $(\frakh, [\cdot, \cdot]_\frakh)$ be Leibniz algebras  with
respect to the Leibniz $\frakg$-representation $(\frakh, \rho^L, \rho^R)$ and  $H: \frakg \rightarrow \frakh$ be a linear map. \\
(Ca) Suppose that $(\frakh, \rho^L_H, \rho^R_H)$ is a Leibniz $\frakg$-representation given by (2). Then the linear map $\hat{H} : \frakg \ltimes_H \frakh\rightarrow
\frakg \ltimes \frakh$ defined by
\begin{eqnarray*}
\hat{H}(x, h)=(x, H(x)+h), \forall x\in \frakg, h\in \frakh,
\end{eqnarray*}
 is a Leibniz algebra isomorphism if and only if $H$ is a crossed homomorphism from $\frakg$ to $\frakh$ with
respect to the  Leibniz $\frakg$-representation $(\frakh, \rho^L, \rho^R)$.\\
(Cb) $H$ is a crossed homomorphism from $\frakg$ to $\frakh$ with
respect to the  Leibniz $\frakg$-representation $(\frakh, \rho^L, \rho^R)$ if and only if  the
map $\iota_H : \frakg \rightarrow \frakg \ltimes_H \frakh$  defined by
\begin{eqnarray*}
\iota_H(x)=(x, H(x)), \forall x\in \frakg
\end{eqnarray*}
is  a Leibniz algebra homomorphism.
\end{thm}
\begin{proof}
(Ca)  Clearly $\hat{H}$ is an invertible linear map. For all $x, y \in \frakg, h, k\in  \frakh$, we have
\begin{eqnarray*}
 [\hat{H}(x, h), \hat{H}(y, k)]&=& [(x, H(x)+h), (y, H(y)+k)]\\
&=& ([x, y]_\frakg, \rho^L(x)(H(y)+k)+\rho^R(y)(H(x)+h)+[H(x)+h, H(y)+k]_\frakh)\\
&=& ([x, y]_\frakg, \rho^L(x)k+\rho^R(y)h+[H(x), k]_\frakh+[h, H(y)]_\frakh+[h, k]_\frakh\\
&& +[H(x), H(y)]_\frakg+\rho^L(x)H(y)+\rho^R(y)H(x))\\
\hat{H} [(x, h), (y, k)]_H&=& ([x, y]_\frakg, H([x, y]_\frakg)+\rho_H^L(x)k+\rho_H^R(y)h+[h, k]_\frakh)\\
&=& ([x, y]_\frakg, H([x, y]_\frakg)+\rho^L(x)k+[H(x), k]_\frakh+\rho^R(y)h+[h, H(y)]_\frakg\\
&& +[h, k]_\frakh).
\end{eqnarray*}
Thus $[\hat{H}(x, h), \hat{H}(y, k)]=\hat{H} [(x, h), (y, k)]_H$ if and only if (\ref{Def: crossed homomorphisms}) holds for $H$, which is equivalent
to that $H$ is a crossed homomorphism from from $\frakg$ to $\frakh$ with
respect to the  Leibniz $\frakg$-representation $(\frakh, \rho^L, \rho^R)$.\\
(Cb) follows from the proof of (Ca) by taking $h = k = 0$.
\end{proof}
\bigskip

\section{Cohomology theory of crossed homomorphisms on Leibniz algebras} \label{Sect: Cohomology theory of Rota-Baxter 3-Lie algebras}
In this  section,  we consider a differential graded Lie algebra (dgLa) whose Maurer-Cartan elements are given by crossed homomorphisms on Leibniz algebras. This characterizations of
a crossed homomorphism  allow us to define cohomology for a crossed homomorphism.

Let $(\frakg, [\cdot, \cdot]_\frakg)$ and $(\frakh, [\cdot, \cdot]_\frakh)$ be two Leibniz algebras and  $(\frakh, \rho^L, \rho^R)$ be a Leibniz $\frakg$-representation. We denote the Leibniz products on $\frakg$ and $\frakh$ respectively by $\mu_\frakg$ and $\mu_\frakh$. Consider the semidirect product
 $\frakg\oplus \frakh$.  Note that $\mu_\frakh$ is a Maurer-Cartan element in the graded Lie algebra $\oplus_n\mathrm{Hom}((\frakg\oplus \frakh)^{\otimes n}, \frakg\oplus \frakh)$. Therefore, we can define a differential $d_{\mu_\frakh}=\llbracket\mu_\frakh, ~\rrbracket$ and the derived bracket on the graded space $\oplus_n\mathrm{Hom}(\frakg^{\otimes n},  \frakh)$ by
 \begin{eqnarray*}
\widehat{\llbracket f, g\rrbracket} := (-1)^{m}\llbracket\llbracket\mu_\frakh,f\rrbracket,g\rrbracket,
\end{eqnarray*}
for any $f\in \mathrm{Hom}(\frakg^{\otimes m},  \frakh)$ and $g\in \mathrm{Hom}(\frakg^{\otimes n},  \frakh)$.

Moreover, we know that $\mu_\frakg+\rho^L+ \rho^R$ is a Maurer-Cartan element in the graded Lie algebra $\oplus_n\mathrm{Hom}((\frakg\oplus \frakh)^{\otimes n}, \frakg\oplus \frakh)$ from \cite{ST1}, Thus it induces a differential $d_{\mu_\frakg+\rho^L+ \rho^R}=\llbracket\mu_\frakg+\rho^L+ \rho^R, ~\rrbracket$, and the graded space $\oplus_n\mathrm{Hom}(\frakg^{\otimes n},  \frakh)$ is closed under the differential $d=d_{\mu_\frakg+\rho^L+ \rho^R}$ and is given by
\begin{align*}
&(d f)(x_1,\ldots, x_{n+1})\\
&=(-1)^{n+1}\sum^{n}_{i=1}(-1)^{i+1}\rho^L(x_i)f(x_1,\ldots, \hat{x}_i, \ldots, x_{n+1}) + \rho^R(x_{n+1})f(x_1,\ldots,  x_{n})\\
&+(-1)^{n+1}\sum_{1\leq i< j\leq n+1} (-1)^{i}f(x_1,\ldots, \hat{x}_i, \ldots, x_{j-1}, [x_i, x_j]_\frakg, x_{j+1},\ldots, x_{n+1}).
\end{align*}

Finally, we have $\llbracket\mu_\frakg+\rho^L+ \rho^R, \mu_\frakh\rrbracket$=0. Hence, $(\oplus_n\mathrm{Hom}(\frakg^{\otimes n}, \frakh),\widehat{\llbracket~, ~\rrbracket}, d)$ is a dgLa.

\begin{prop}
 A linear map  $H : \frakg\rightarrow  \frakh$ is a crossed homomorphism from $\frakg$ to $\frakh$ if
and only if $H \in C^1(\frakg, \frakh)$ is a Maurer-Cartan element in the dgLa $(\oplus_n\mathrm{Hom}(\frakg^{\otimes n}, \frakh),\widehat{\llbracket~, ~\rrbracket}, d)$.
\end{prop}
\begin{proof}
For any linear map $H: \frakg\rightarrow \frakh$ and $x, y\in \frakg$, we have
\begin{eqnarray*}
(dH+\frac{1}{2}\widehat{\llbracket H, H\rrbracket})(x, y)=\rho^L(x)H(y)+\rho^R(y)H(x)-H([x, y]_\frakg)+[H(x), H(y)]_\frakh.
\end{eqnarray*}
Hence $H$ is a crossed homomorphism if and only if $H$ is a Maurer-Cartan element.
\end{proof}
It follows from the above proposition that a crossed homomorphism $H$ induces a differential $d_H = d + \widehat{\llbracket H, ~\rrbracket}$
 on the graded Lie algebra $(\oplus_n\mathrm{Hom}(\frakg^{\otimes n}, \frakh),\widehat{\llbracket~, ~\rrbracket})$. Define $C^n(\frakg, \frakh)=\mathrm{Hom}(\frakg^{\otimes n}, \frakh)$
and $C^{\ast}(\frakg, \frakh)=\oplus_n C^n(\frakg, \frakh)$. Thus a
crossed homomorphism induces a dgLa $(C^{\ast}(\frakg, \frakh), \widehat{\llbracket~, ~\rrbracket}, d_H)$.

\begin{thm}
Let $(\frakg, [\cdot, \cdot]_\frakg)$ and $(\frakh, [\cdot, \cdot]_\frakh)$ be two Leibniz algebras and  $(\frakh, \rho^L, \rho^R)$ be a Leibniz $\frakg$-representation. Suppose $H$ is a crossed homomorphism. Then for any linear map $H': \frakg\rightarrow \frakh$, the sum $H+H'$ is also a crossed homomorphism if and only if $H'$ is a  Maurer-Cartan element in the dgLa $(C^{\ast}(\frakg, \frakh), \widehat{\llbracket~, ~\rrbracket}, d_H)$.

\end{thm}
\begin{proof}
\begin{eqnarray*}
&&(d(H+H')+\frac{1}{2}\widehat{\llbracket H+H', H+H'\rrbracket})\\
&&= dH+dH'+\frac{1}{2}(\widehat{\llbracket H, H\rrbracket}+\widehat{\llbracket H, H'\rrbracket}+\widehat{\llbracket H', H\rrbracket}+\widehat{\llbracket H', H'\rrbracket})\\
&&= dH'+\widehat{\llbracket H, H'\rrbracket}+\frac{1}{2}\widehat{\llbracket H', H'\rrbracket}\\
&&= d_H(H')+\frac{1}{2}\widehat{\llbracket H', H'\rrbracket}.
\end{eqnarray*}
And the proof is finished.
\end{proof}

The cohomology of the cochain complex $(C^{\ast}(\frakg, \frakh), d_{H})$ is called the cohomology of the crossed homomorphism $H$,  if $\mathcal{Z}_H^k(\frakg, \frakh) =\{ f \in C^k(\frakg, \frakh)| d_H(f) = 0\}$ is the space of $k$-cocycles and $\mathcal{B}^k_H(\frakg, \frakh) =\{d_H (f)\in  C^k(\frakg, \frakh)| f \in C^{k-1}(\frakg, \frakh)\}$ is the space of $k$-coboundaries then $\mathcal{B}^k_H(\frakg, \frakh) \subset \mathcal{Z}^k_H(\frakg, \frakh)$, for $k \geq 0$. The corresponding cohomology groups
\begin{eqnarray*}
\mathcal{H}^k_H(\frakg, \frakh)=\frac{\mathcal{Z}^k_H(\frakg, \frakh)}{\mathcal{B}^k_H(\frakg, \frakh)}, k \geq 0.
\end{eqnarray*}
We denote the corresponding cohomology groups simply by $\mathcal{H}^{\ast}(\frakg, \frakh)$.

First recall from Lemma (\ref{new-equations}) that a crossed homomorphism $H$ induces a   Leibniz $\frakg$-representation given by
\begin{eqnarray*} \label{new-equations}
&&\rho^L_H(x)h:=\rho^L(x)h+[H(x), h]_\frakh,~~~~~~~~ \rho^R_H(x)h:=\rho^R(x)h+[h, H(x)]_\frakh.
\end{eqnarray*}
The corresponding  cochain groups are given by $C^n_{\mathrm{Leib}}(\frakg, \frakh) =\mathrm{Hom}(\frakg^{\otimes n}, \frakh)$, for $n \geq 0$, and the coboundary operator $\delta_{\mathrm{Leib}}: C^n_{\mathrm{Leib}}(\frakg, \frakh)\rightarrow C^{n+1}_{\mathrm{Leib}}(\frakg, \frakh)$ is given by
\begin{eqnarray*}
&&(\delta_{\mathrm{Leib}} f)(x_1,\ldots, x_{n+1})\\
&=&\sum^{n}_{i=1}(-1)^{i+1}\rho^L(x_i)f(x_1,\ldots, \hat{x}_i, \ldots, x_{n+1})+[H(x_i), f(x_1,\ldots, \hat{x}_i, \ldots, x_{n+1})]_\frakh \\
&& + (-1)^{n+1}\rho^R(x_{n+1})f(x_1,\ldots,  x_{n})+(-1)^{n+1}[f(x_1,\ldots,  x_{n}), H(x)]_\frakh\\
&&+\sum_{1\leq i< j\leq n+1} (-1)^{i}f(x_1,\ldots, \hat{x}_i, \ldots, x_{j-1}, [x_i, x_j]_\frakg, x_{j+1},\ldots, x_{n+1}),
\end{eqnarray*}
for $x_1, \ldots, x_{n+1}\in \frakg$.

\begin{prop}
The coboundary operators $d_H$ and $\delta_{\mathrm{Leib}}$ are related by
\begin{eqnarray*}
d_H(f)=(-1)^{n-1}\delta_{\mathrm{Leib}}(f), \forall f\in C^n(\frakg, \frakh).
\end{eqnarray*}
\begin{proof}
For any $f\in C^n(\frakg, \frakh)$, we have
\begin{eqnarray*}
&&(-1)^{n-1}(d_H(f))(x_1,\ldots, x_{n+1})\\
&=& (-1)^{n-1}(df+\widehat{\llbracket H, f\rrbracket})(x_1,\ldots, x_{n+1})\\
&=& \sum^{n}_{i=1}(-1)^{i+1}\rho^L(x_i)f(x_1,\ldots, \hat{x}_i, \ldots, x_{n+1}) + (-1)^{n+1}\rho^R(x_{n+1})f(x_1,\ldots,  x_{n})\\
&&+\sum_{1\leq i< j\leq n+1} (-1)^{i}f(x_1,\ldots, \hat{x}_i, \ldots, x_{j-1}, [x_i, x_j]_\frakg, x_{j+1},\ldots, x_{n+1})\\
&& +\{[H(x_i), f(x_1,\ldots, \hat{x}_i, \ldots, x_{n+1})]_\frakh+ (-1)^{n+1}[f(x_1,\ldots,  x_{n}), H(x)]_\frakh\}\\
&=& \delta_{\mathrm{Leib}}(f).
\end{eqnarray*}
And the proof is finished.
\end{proof}
\end{prop}
\bigskip


\section{Deformations of crossed homomorphisms on Leibniz algebras}

In this section, we study deformations of a crossed homomorphism on Leibniz algebras.

\subsection{Linear deformations}

Let $H: \frakg\rightarrow \frakh$  be a crossed homomorphism. A linear deformation of $H$ consists
of a sum $H_t = H + tH_1$ such that $H_t$
is a crossed homomorphism, for all values
of $t$. In such a case, we say that $H_1$ generates a one-parameter linear deformation
of $H$.
Thus, if $H_1$ generates a linear deformation of $H$ then $H_t = H + tH_1$
satisfies
\begin{eqnarray*}
H_t([x, y]_\frakg)=\rho^L(x)H_t(y)+\rho^L(y)H_t(x)+[H_t(x), H_t(y)]_\frakg, \forall x, y\in \frakg.
\end{eqnarray*}
That is
\begin{eqnarray}
&& H_1([x, y]_\frakg) = \rho^L(x)H_1(y)+\rho^L(y)H_1(x)+[H_1(x), H(y)]_\frakg+[H(x), H_1(y)]_\frakg,\\
&& [H_1(x), H_1(y)]_\frakg = 0.
\end{eqnarray}

It  follows from (3) that $H_1$ is a 1-cocycle in the cohomology of the crossed homomorphism $H$.

\begin{defn}
Two linear deformations $H_t = H + tH_1$ and $H'_t = H + tH'_1$
are said to be equivalent if there exists $x\in \frakg$ such that
\begin{eqnarray*}
&& \phi_t = \mathrm{id}_\frakg + tL_x,  \quad \psi_t=\mathrm{id}_\frakh + t\rho^{L}(x)
\end{eqnarray*}
is a homomorphism  from $H_t$ to $H'_t$.
\end{defn}

Thus, if  $H_t$ and $H'_t $ are equivalent linear deformations, then  the following conditions hold:
\begin{eqnarray*}
&&(i)~~ [\phi_t(y), \phi_t(z)]_\frakg=\phi_t([y, z]_\frakg),~~~[\psi_t(h), \psi_t(k)]_\frakh=\psi_t([h, k]_\frakh),\\
&&(ii)~~\phi_t (\rho^{L}(y)h)=\rho^{L}(\phi_t(y))\psi_t(h),\\
&&(iii)~~\phi_t (\rho^{R}(y)h) =\rho^{R}(\phi_t(y))\psi_t(h),\\
&&(iv)~~H'_t\circ \phi_t (y)= \psi_t \circ H_t (y),
\end{eqnarray*}
for all $y, z\in \frakg$ and $h, k \in \frakh$.

Note that (i) is equivalent
\begin{eqnarray}
[[x, y]_\frakg, [x, z]_\frakg]_\frakg = 0,~~~~~  [\rho^{L}(x)h, \rho^{L}(x)k]_\frakh = 0.
\end{eqnarray}

Further, (ii)and (iii) imply that
\begin{eqnarray}
&& \rho^L([x, y]_\frakg)\rho^L(x)=0, \\
&&\rho^R([x, y]_\frakg)\rho^L(x)=0, ~~~~\forall y\in \frakg.
\end{eqnarray}

Finally, the condition (iv) is equivalent to
\begin{eqnarray}
&& H_1(y)-H'_1(y)=\rho^{R}(y)H(x)+[H(x), H(y)]_\frakh,\\
&& \rho^{L}(x)H_1(y)=H_1'([x, y]_\frakg).
\end{eqnarray}

It follows from (8) that $H_1(y)-H'_1(y)=\delta_{\mathrm{Leib}}(-H(x))y.$ Hence, we obtain the following
\begin{thm}
Let $H_t = H + tH_1$ be a linear deformation of $H$. Then the linear term
$H_1$ is a 1-cocycle in the cohomology of $H$. Its cohomology class depends only on the
equivalence class of $H_t$.
\end{thm}

\begin{defn}
A linear deformation $H_t$ of a crossed homomorphism $H$ is said to
be trivial if $H_t$ is equivalent to $H'_t= H$.
\end{defn}

\begin{defn}
Let $H$ be a crossed homomorphism from $\frakg$ to $\frakh$. An element
$x\in \frakg$ is called a Nijenhuis element associated to $H$ if $x$ satisfies (5), (6), (7) and
\begin{eqnarray*}
[x, \rho^{R}(y)H(x)+[H(x), H(y)]_\frakh]_\frakg=0,~~~~\forall y\in \frakg.
\end{eqnarray*}
\end{defn}
Denote by $Nij(H)$ the set of Nijenhuis elements associated to $H$. Then we have

\begin{thm}
Let $H$ be a crossed homomorphism from $\frakg$ to $\frakh$.. Then for any $x\in  Nij(H)$,
the sum $H_t = H + tH_1$ with $H_1 = -\delta_{\mathrm{Leib}}(H(x))$ is a trivial deformation of $H$.
\end{thm}
\bigskip
\subsection{Formal deformations}

Let $H: \frakg\rightarrow \frakh$  be a crossed homomorphism. Let $\mathbf{k}[[t]]$ be the ring of power series in one variable $t$. For any $\mathbf{k}$-linear space $\frakg$,
let $\frakg [[t]]$ denotes the vector space of formal power series in $t$ with coefficients from $\frakg$. Then  $\frakg[[t]]$ is Leibniz
algebra structure over $\mathbf{K}[[t]]$.  If $\frakh$ is a Leibniz algebra which is also a Leibniz $\frakg$-representation,  then $\frakh[[t]]$ is a Leibniz algebra over $\mathbf{k}[[t]]$ and a Leibniz $\frakg [[t]]$-representation.
\begin{defn}
 A formal one-parameter deformation of a crossed homomorphism
$H: \frakg\rightarrow \frakh$  consists of a formal sum
\begin{eqnarray}
H_t = H_0 + tH_1 + t^2H_2 + \cdots \in \mathrm{ Hom}(\frakg, \frakh)[[t]]
\end{eqnarray}
with $H_0 = H$ such that $H_t: \frakg[[t]] \rightarrow \frakh[[t]]$ is a crossed homomorphism from
$\frakg[[t]]$ to $\frakh[[t]]$.
\end{defn}
Note that (10) is  equivalent to: for each $n\geq 0$
\begin{eqnarray*}
H_n([x, y]_\frakg)=\rho^L(x)H_n(y)+\rho^R(y)H_n(x)+\sum_{i+j=n}[H_i(x), H_j(y)]_\frakh.
\end{eqnarray*}
That is
\begin{eqnarray*}
d_H(H_n)=-\frac{1}{2} \sum_{i+j=n, i, j\geq 1} \widehat{\llbracket H_i, H_j\rrbracket}, n\geq 0.
\end{eqnarray*}
The identity holds for $n = 0$ as $H$ is a crossed homomorphism. For $n = 1$, we get
$d_H(H_1) = 0$. Hence $H_1$ is a 1-cocycle in the cohomology complex of $H$. This is
called the infinitesimal of the deformation $H_t$.

\begin{defn}
Two o formal deformations $H_t$ and $H'_t$ of a crossed  homomorphism $H$
are said to be equivalent if there exists $x\in \frakg$ such that
\begin{eqnarray*}
&& \phi_t = \mathrm{id}_\frakg + tL_x,  \quad \psi_t=\mathrm{id}_\frakh + t\rho^{L}(x)
\end{eqnarray*}
is a homomorphism  from $H_t$ to $H'_t$.
\end{defn}
\begin{prop}
Let $H_t$ be a formal deformation of a crossed homomorphism $H$. Then
the linear term $H_1$ is a 1-cocycle in the cohomology of $H$. (It is called the infinitesimal
of the deformation.) Moreover, the corresponding cohomology class depends only on the
equivalence class of $H_t$.
\end{prop}
\begin{defn}
 Let $(\frakg, [\cdot, \cdot]_\frakg)$ and $(\frakh, [\cdot, \cdot]_\frakh)$ be two Leibniz algebras and $(\frakh, \rho^L, \rho^R)$ a Leibniz $\frakg$-representation. A crossed homomorphism $H : \frakg \rightarrow \frakh$ is said to be rigid if any
formal deformation $H_t$ of $H$ is equivalent to $H'_t = H$.
\end{defn}
\begin{thm}
Let $H$ be a crossed homomorphism. If $\mathcal{Z}^1_H(\frakg, \frakh)=\delta_{\mathrm{Leib}}(H(Nij(H)))$,  then $H$ is rigid.
\end{thm}
\begin{proof}
 Let $H_t = \sum_{i\geq 0}t^iH_i$ be any formal deformation of $H$.   By Proposition 4.8, we deduce $H_1\in \mathcal{ Z}^1(\frakg, \frakh)$. By the assumption
$\mathcal{Z}^1_H(\frakg, \frakh)=\delta_{\mathrm{Leib}}(H(Nij(H)))$, we obtain $H_1 = -\delta_{\mathrm{Leib}}(H(x))$ for some $x \in Nij(H)$. Then setting
$ \phi_t = \mathrm{id}_\frakg + tL_x,  \quad \psi_t=\mathrm{id}_\frakh + t\rho^{L}(x)$, we get a formal deformation $\overline{H}_t := \psi_t\circ H_t \phi_t^{-1}$
Thus, $H_t$ is equivalent to $\overline{H}_t$. Moreover, we have
\begin{eqnarray*}
\overline{H}_t(y)) (\mathrm{mod} ~t^2) &=&  (\mathrm{id}_\frakg+ t\rho^L(x)) \circ (H + tH_1)(y - t[x, y]_\frakg) (\mathrm{mod} ~t^2)\\
&=&(\mathrm{id}_\frakg+ t\rho^L(x)) (H(y) + tH_1(y) - tH([x, y]_\frakg)) (\mathrm{mod}~ t^2)\\
&=& H(y)+t(\rho^L(x)H(y)-H([x, y]_\frakg+H_1(y)).
\end{eqnarray*}
Since
\begin{eqnarray*}
H_1(y)&=& -\delta_{\mathrm{Leib}}(H(x))(y)\\
&=& \rho^{R}(y)H(x)+[H(x), H(y)]_\frakh\\
&=& -(\rho^L(x)H(y)-H([x, y]_\frakg).
\end{eqnarray*}
The coefficient of $t$ in the expression of $\overline{H}_t$
is zero. Then by repeating this argument, one get the equivalence between
$H_t$ and $H$. Hence the proof.
\end{proof}
\bigskip

\subsection{Extensions of finite order deformations}

In this subsection, we introduce a cohomology class associated to any order $N$ deformation of a crossed homomorphism, and show that an order $n$ deformation is extensible if and only if this cohomology class is trivial. Thus, we call this cohomology class the obstruction class of the order $N$ deformation being extensible.

Let $H: \frakg\rightarrow \frakh$  be a crossed homomorphism.  Consider the space $\frakg[[t]]/(t^{N+1})$
which inherits a Leibniz algebra structure over $\mathbf{k}[[t]]/(t^{N+1})$, similarly, $\frakh[[t]]/(t^{N+1})$ is a Leibniz algebra and a Leibniz $\frakg[[t]]/(t^{N+1})$-representation.

\begin{defn}
A deformation of $H$ of order $N$ consists of a finite sum $H_t = \sum_{i=0}^N
t^iH_i$ with $H_0 = H$ such that $H_t$ is a crossed homomorphism from  $\frakg[[t]]/(t^{N+1})$ to $\frakh[[t]]/(t^{N+1})$.
\end{defn}
\begin{defn}
 If  there exists a  linear map $H_{N+1}: \frakg\rightarrow \frakh$ such that $\hat{H}_t = H_t +t^{N+1}H_{N+1}$ is a deformation of order $N+ 1$, we say that $H_t$ is extensible.
\end{defn}

Thus, if a deformation $H_t$ of order $N$ is extensible then one more deformation
equation need to be satisfied, namely,
\begin{eqnarray*}
d_H(H_{N+1})=-\frac{1}{2} \sum_{i+j=N+1, i, j\geq 1} \widehat{\llbracket H_i, H_j\rrbracket}, n\geq 0.
\end{eqnarray*}
for some $H_{N+1}$. Note that the right hand side of the above equation does not
involve $H_{N+1}$,  we  denote it by $Ob_{H_t}$. Obviously $Ob_{H_t}$
is a 2-cochain in the cohomology of $H$,  we have the following.
\begin{prop}
The $2$-cochain $Ob_{H_t}$ is a 2-cocycle, that is, $d_{H}(Ob_{H_t})= 0$.
\end{prop}
\begin{proof}We have
\begin{eqnarray*}
&&d_{H}(Ob_{H_t})\\
&=&-\frac{1}{2}\sum_{i+j=N+1, i, j\geq 1}(d \widehat{\llbracket H_i, H_j\rrbracket}+ \widehat{\llbracket H, \widehat{\llbracket H_i, H_j\rrbracket} \rrbracket})\\
&=& -\frac{1}{2}\sum_{i+j=N+1, i, j\geq 1}(\widehat{\llbracket d(H_i), H_j\rrbracket}-\widehat{\llbracket H_i, d(H_j)\rrbracket}+\widehat{\llbracket \widehat{\llbracket H, H_i\rrbracket}, H_j\rrbracket}-\widehat{\llbracket H_i, \widehat{\llbracket H, H_j\rrbracket} \rrbracket})\\
&=& -\frac{1}{2}\sum_{i+j=N+1, i, j\geq 1}(\widehat{\llbracket d_H(H_i), H_j\rrbracket}-\widehat{\llbracket H_i, d_H(H_j)\rrbracket})\\
&=& \frac{1}{4}\sum_{i_1+i_2+j=N+1, i_1,i_2, j\geq 1}\widehat{\llbracket \widehat{\llbracket H_{i_1}, H_{i_2}\rrbracket}, H_j\rrbracket} -\frac{1}{4}\sum_{i+j_1+j_{2}=N+1, i_1,i_2, j\geq 1} \widehat{\llbracket H_i, \widehat{\llbracket H_{j_1}, H_{j_2}\rrbracket} \rrbracket}\\
&=& \frac{1}{2}\sum_{i+j+k=N+1, i, j, k\geq 1}\widehat{\llbracket \widehat{\llbracket H_{i}, H_{j}\rrbracket}, H_k\rrbracket}\\
&=&0.
\end{eqnarray*}
\end{proof}
The cohomology class $[Ob_{H_t}] \in \mathcal{H}_H^2(\frakg, \frakh)$ is called the {\bf obstruction class} to
extend the deformation $H_t$.  we have the
following.
\begin{thm} A finite order deformation $H_t$ of a crossed homomorphism $H$ extends
to a deformation of next order if and only if the obstruction class $[Ob_{H_t}] \in \mathcal{H}_H^2(\frakg, \frakh)$
vanishes.
\end{thm}
\begin{cor}
If $\mathcal{H}_H^2(\frakg, \frakh) = 0$,  then any finite order deformation of $H$ extends to a
deformation of next order.
\end{cor}
\begin{thm} If $\mathcal{H}_H^2(\frakg, \frakh) = 0$, then every 1-cocycle in the cohomology of $H$ is the
infinitesimal of some formal deformation of $H$.
\end{thm}

\bigskip

\noindent
{{\bf Acknowledgments.} The work is supported by Natural Science Foundation of China (Grant Nos. 11871301, 12271292).

\end{document}